\documentclass[12pt]{article}
\usepackage{latexsym,amsmath,amssymb}

\newcommand{\fl}{\longrightarrow}
\newfont{\bb}{msbm10 at 12pt}

\def\h{\hbox{\bb H}}

\def\s{\hbox{\bb S}}

\def\mm{\mathbb{M}\times\mathbb{R}}
\def\m{\mathbb{M}}
\def\r{\mathbb{R}}

\newcommand{\ee}{\begin{equation}}

\newcommand{\fe}{\end{equation}}

\usepackage[latin1]{inputenc}
\usepackage {amsmath}
\usepackage {amsthm}
\usepackage{times}
\usepackage{amscd}
\usepackage{epsf}

\begin{document}

\theoremstyle{plain}\newtheorem{lem}{Lemma}
\theoremstyle{plain}\newtheorem{pro}{Proposition}
\theoremstyle{plain}\newtheorem{teo}{Theorem}
\theoremstyle{plain}\newtheorem{eje}{Example}[section]
\theoremstyle{plain}\newtheorem{no}{Remark}
\theoremstyle{plain}\newtheorem{cor}{Corollary}

\begin{center}
\rule{15cm}{1.5pt} \vspace{.6cm}

{\Large \bf Geometric barriers for the existence of hypersurfaces\\[3mm] with prescribed curvatures in M$^n\times$R.}\\[5mm]

\rule{15cm}{1.5pt}\vspace{.5cm}\\

Jos\'e A. G\'alvez$^{a\ }$\footnote{The first author is partially supported by MICINN-FEDER, Grant No. MTM2010-19821 and by Junta de Andaluc\'{\i}a Grant No. FQM325 and P09-FQM-5088.}, Victorino Lozano$^b$\vspace{0.4cm} \vspace{0.4cm}
\end{center}


\noindent $\mbox{}^a$ Departamento de Geometr\'{\i}a y Topolog\'{\i}a, Universidad de Granada,
E-18071 Granada, Spain; e-mail: jagalvez@ugr.es \vspace{0.2cm}

\noindent $\mbox{}^b$ Departamento de Matem\'atica Aplicada, U.N.E.D. E-13600 Alc\'azar de San Juan, Ciudad Real, Spain;
e-mail: vlozano@valdepenas.uned.es \vspace{0.3cm}

\begin{abstract}
We show the existence of a deformation process of hypersurfaces from a product space $\m_1\times\r$ into another product space $\m_2\times\r$ such that the relation of the principal curvatures of the deformed hypersurfaces can be controlled in terms of the sectional curvatures or Ricci curvatures of $\m_1$ and $\m_2$. In this way, we obtain barriers which are used for proving existence or non existence of hypersurfaces with prescribed curvatures in a general product space $\mm$.
\end{abstract}


\section{Introduction.}
Our main objective is to describe a simple method for obtaining barriers in
a product space $\m^n\times\r$. To that end, we consider a hypersurface $S$ in a product space $\m_1\times\r$ and obtain a new hypersurface $S^\ast$ in a different product space $\m_2\times\r$ such that the principal curvatures of $S$ and $S^\ast$ can be related in terms of the sectional curvatures or Ricci curvatures of $\m_1$ and $\m_2$.

The previous method has a special interest when $S$ is a hypersurface with constant mean curvature $H$, or in general with constant r-mean curvature $H_r$, and $\m_1$ has constant sectional curvature $c$. In such a case, we obtain barriers for the existence of hypersurfaces with constant r-mean curvature $H_r$ in a general $\m_2\times\r$ if the sectional (or Ricci) curvature of $\m_2$ is bounded from above or below by $c$. 

Thus, we will generalize different known results in the homogeneous spaces $\h^n\times\r$, $\r^{n+1}$ or $\s^n\times\r$ to general product spaces $\m^n\times\r$. For the case $n=2$ the analogous method was described in \cite{GL}.

The paper is organized as follows. Section \ref{s2} is devoted to the calculus of the principal curvatures of certain graphs in a general product $\m^n\times\r$ in terms of Jacobi fields on $\m^n$ (Lemma \ref{l1}). In Section \ref{s3} we describe in detail our method for obtaining barriers in a product space $\m^n\times\r$, and obtain two comparison results. The first one relates the principal curvatures of two graphs $S$ in $\m_1\times\r$ and $S^{\ast}$ in $\m_2\times\r$ with the sectional curvatures of $\m_1$ and $\m_2$ (Theorem \ref{t1}). The second comparison result relates the mean curvatures of $S$ and $S^{\ast}$ with the Ricci curvatures of $\m_1$ and $\m_2$ (Theorem \ref{t2}). 

In Section 4 we show different examples of how these barriers can be used. Thus, in Theorem \ref{t3} we prove that given a closed geodesic ball $B_r\subseteq \m^n$ of radius $r$ then there is an explicit constant $H_0$, which only depends on the radius $r$ and the minimum of its Ricci curvature, such that there exists no vertical graph over $B_r$ in $\m^n \times \r$ with minimum of its mean curvature greater than or equal to $H_0$. This generalizes a previous result by Espinar and Rosenberg in \cite{ER} for $n=2$. In Theorem \ref{t4} we obtain an analogous result for Gauss-Kronecker curvature, or in general for r-mean curvature, depending on the sectional curvatures of the closed geodesic ball. 

Moreover, in Theorem \ref{t5},  we prove that, under certain restrictions on the ambient space $\m^n \times \r$, for every properly embedded hypersurface $\Sigma \subset \m^n \times \r$ with mean curvature $H \geq H_0 >0$, its mean convex component cannot contain a certain geodesic ball of radius $r$, where $r$ only depends on $H_0$ and the infimum of the Ricci curvature of $\m$. In particular, this shows the non existence of entire horizontal graphs over  a Hadamard manifold for certain values of the mean curvature (Corollary \ref{c1}). 	 

In Theorem \ref{t7} we also prove the existence of vertical graphs in $\m^n \times \r$ with boundary on a horizontal slice and constant mean curvature $H_0$ for any $H_0\in[0,(n-1)/n]$, when $\m^n$ is a Hadamard manifold with sectional curvature pinched between $-c^2$ and $-1$. In fact, we show that any compact hypersurface with constant mean curvature $H_0$ and the same boundary must be the previous graph or its reflection with respect to the slice. This generalizes previous results in $\h^n\times\r$ (see \cite{NSST} and \cite{BE2}).

Finally, in Theorem \ref{t8} we give a result of existence for vertical graphs with positive constant Gauss-Kronecker curvature in $\m^n \times \r$, which solves the Dirichlet problem for the associated Monge-Amp\`ere equation with zero boundary values.

\section{The principal curvatures of the graph.}\label{s2}

Let ${\cal H}$ be an (n-1)-dimensional manifold and $(\m,g)$ be an $n$-dimensional Riemannian manifold. Consider two smooth maps $i:{\cal H}\fl\m$ (non necessarily an immersion) and $n:{\cal H}\fl T\m$ such that $n(x)\in T_{i(x)}\m$ is a unit vector with $g(di_x(v),n(x))=0$ for all $v\in T_x{\cal H}$. Here, for instance, $T_p\m$ denotes the tangent space to $\m$ at the point $p\in\m$.

Let $I$ be an open real interval such that 0 is in its closure $\bar{I}$, and assume that the map
\begin{equation}\label{eliminado}
\varphi(x,t)=\exp_{i(x)}(t\ n(x)),\qquad (x,t)\in {\cal H}\times \bar{I},
\end{equation}
is smooth and $\varphi_{|{\cal H}\times I}$ a global diffeomorphism onto its image; where $\exp$ denotes the exponential map in $\m$.

Observe that $\varphi_{|{\cal H}\times I}$ can be seen as a certain parametrization of an open set of $\m$, and the parameter $t$ can be considered as a distance function to  $i(x)$.

The polar geodesic parameters at a point $p\in\m$ are examples of the previous situation. For that, one can consider ${\cal H}$ as the unit sphere of $T_p\m$, $i$ as the constant map $i(x)=p$ and $n(x)=x$.

Now, let us consider the product space $\mm$ with the standard product metric and let us call $h$ to the parameter in $\r$. Let $\psi(x,t)$ be the graph given by the height function $f(t)$ which only depends on the distance function, that is, the graph in $\mm$ parameterized as
\begin{equation}\label{psi}
\psi(x,t)=(\exp_{i(x)}(t\ n(x)),f(t))=(\varphi(x,t),f(t)).
\end{equation}

Then, one has
$$
\begin{array}{l}
\overline{\partial}_{x_i}=\partial_{x_i},\qquad i=1,\ldots,n-1\\
\overline{\partial}_t=\partial_t+f'(t)\partial_h,
\end{array}
$$
where $x=(x_1,\ldots,x_{n-1})$ are local coordinates in ${\cal H}$. Here, for instance, $\partial_{x_i}$ denotes the vector field $\frac{\partial\ }{\partial x_i}$ in $\mm$ and $\overline{\partial}_{x_i}$ the corresponding vector field in the graph.

If $\langle,\rangle=g+dh^2$ stands for the product metric in  $\mm$, then from the Gauss lemma we obtain
$$
\langle\overline{\partial}_{x_i},\overline{\partial}_{t}\rangle=g(\partial_{x_i},\partial_{t})=0.
$$

Hence, the pointing upwards unit normal of the graph is
$$
N=\frac{1}{\sqrt{1+f'(t)^2}}(-f'(t)\partial_t+\partial_h).
$$

So, if  we denote by $\overline{\nabla}$ the Levi-Civita connection in $\mm$, it is easy to see that
$$
-\overline{\nabla}_{\overline{\partial}_{t}}N=\frac{f''(t)}{(1+f'(t)^2)^{3/2}}\ \overline{\partial}_{t}.
$$
In particular, $\overline{\partial}_{t}$ is a principal direction with associated principal curvature
$$
k_n=\frac{f''(t)}{(1+f'(t)^2)^{3/2}}.
$$
Observe that this principal curvature does not depend on either ${\cal H}$, or $\m$, or its metric.

In  order to compute the rest of principal curvatures of the graph we will focus on the  directions which are orthogonal to $\overline{\partial}_{t}$, that is, the ones generated by $\overline{\partial}_{x_i}$.

Let $\gamma(t)=\varphi(x_0,t)$ be a geodesic in $\m$ and $J(t)$ a Jacobi field along $\gamma(t)$ with $g(J(t),\gamma'(t))=0$. If we denote by $\nabla$ the Levi-Civita connection in $\m$ then the second fundamental form of the graph satisfies
$$
II(J,J)=\langle-\overline{\nabla}_J N,J\rangle=g(-\nabla_J\left(\frac{-f'(t)}{\sqrt{1+f'(t)^2}}\partial_t\right),J)=\frac{f'(t)}{\sqrt{1+f'(t)^2}}\ g(\nabla_J\partial_t,J).
$$

On the other hand, since $J$ is a Jacobi field then
$$
\frac{D^2J}{dt^2}+R(J,\gamma')\gamma'=0,
$$
where, as usual, we use the notation $\frac{DJ}{dt}$ for $\nabla_{\gamma'(t)}J$ and $R(X,Y)Z=\nabla_X\nabla_YZ-\nabla_Y\nabla_XZ-\nabla_{[X,Y]}Z$.

Moreover, since $\frac{DJ}{dt}=\nabla_J\partial_t$ because $J$ is a Jacobi field, we obtain that
\begin{eqnarray*}
g(\nabla_J\partial_t,J )\left|_{t_0}\right.&=&g( J(0),\frac{DJ}{dt}(0))+\int_0^{t_0}\frac{d\ }{dt}g(\frac{DJ}{dt},J )\ dt\\
&=&g( J(0),\frac{DJ}{dt}(0) )+\int_0^{t_0}\left(\left|\frac{DJ}{dt}\right|^2-g(R(J,\gamma')\gamma',J)\right)\ dt.
\end{eqnarray*}

Observe that since $\partial_t(x,0)=n(x)$ we have that if $i$ is an immersion then
\begin{equation}\label{e1}
g( J(0),\frac{DJ}{dt}(0) )=g( \nabla_J\partial_t,J )(0)=-II_{\cal H}(J(0),J(0)),
\end{equation}
where $II_{\cal H}$ will denote the second fundamental form of $i:{\cal H}\fl\m$ in the direction of $n$. On the other hand, the amount $g( J(0),\frac{DJ}{dt}(0) )$ vanishes if $i$ is constant (as in the polar geodesic coordinates).

Given a vector field $V(t)$ along $\gamma_{|[0,t_0]}$, with $g(V,\gamma'(t))=0$, we will define the {\it index form} of $V$ as
\begin{equation}\label{e2}
{\cal I}_{(t_0,{\cal H})}(V,V)=g(V(0),\frac{DV}{dt}(0) )+\int_0^{t_0}\left(\left|\frac{DV}{dt}\right|^2-g(R(V,\gamma')\gamma',V)\right)\ dt.
\end{equation}

With all of this, we obtain
\begin{lem}\label{l1}
In the previous conditions the graph $\psi(x,t)$ given by (\ref{psi}) has a principal curvature $k_n$, with principal direction $\overline{\partial}_t$, given by
$$
k_n=\frac{f''(t)}{(1+f'(t)^2)^{3/2}}.
$$
Moreover, the second fundamental form of the graph at $\psi(x_0,t_0)$ for a tangent vector $v_0$ perpendicular to $\overline{\partial}_t$ can be computed as follows: Take a perpendicular Jacobi field $J(t)$ along the geodesic $\gamma(t)=\varphi (x_0,t)$ with $J(t_0)=v_0$
then the second fundamental form is given by
\begin{equation}\label{e3}
II(v_0,v_0)=II(J(t_0),J(t_0))=\frac{f'(t_0)}{\sqrt{1+f'(t_0)^2}}\ {\cal I}_{(t_0,{\cal H})}(J,J).
\end{equation}
\end{lem}

\section{Comparison results.}\label{s3}

The above lemma will help us to compare the principal curvatures of two graphs $\psi_j(x,t)$, $j=1,2$, with the same height function in two different product spaces $\m_j\times\r$. For that, we need to relate the index forms in the manifolds $\m_1$ and $\m_2$.

Thus, let $(\m_1,g_1)$, $(\m_2,g_2)$ be two Riemannian manifolds with $dim(\m_1)\leq dim(\m_2)$ and ${\cal H}_j$, $j=1,2,$ two smooth manifolds with $dim({\cal H}_j)= dim(\m_j)-1$.  Consider smooth maps $i_j:{\cal H}_j\fl\m_j$  and $n_j:{\cal H}_j\fl T\m_j$ such that $n_j(x)\in T_{i_j(x)}\m_j$ is a unit vector with $g(d(i_j)_x(v),n_j(x))=0$ for all $v\in T_x {\cal H}_j$.

Moreover, assume that the maps
$$
\varphi_j(x,t)=\exp_{i_j(x)}(t\ n_j(x)),\qquad (x,t)\in {\cal H}_j\times \bar{I},
$$
are smooth and ${\varphi_j}_{|{\cal H}_j\times I}$ a global diffeomorphism onto its image, where  $I$ is an open real interval such that 0 is in its closure $\bar{I}$. Then, we consider the two graphs
\begin{equation}\label{psi2}
\psi_j(x,t)=(\exp_{i_j(x)}(t\ n_j(x)),f(t))=(\varphi_j(x,t),f(t)),
\end{equation}
for the same height function $f(t)$ with $f'(t)\geq0$.

In order to compare the second fundamental forms of both graphs, let $\overline{\i}:{\cal H}_1\fl {\cal H}_2$ be an immersion, $x_0\in {\cal H}_1$ and $\gamma_j:[0,t_0]\fl\m_j$ be the geodesics
$$
\gamma_1(t)=\varphi_1(x_0,t),\qquad \gamma_2(t)=\varphi_2(\overline{\i}(x_0),t).
$$
\begin{teo}\label{t1}
In the previous conditions assume $K^1_{\gamma_1(t)}(\pi_1)\leq K^2_{\gamma_2(t)}(\pi_2)$, $t\in[0,t_0]$, for each planes $\pi_j$, where $K^j_{\gamma_j(t)}(\pi_j)$ is the sectional curvature of a plane $\pi_j$ in $\m_j$ containing $\gamma'_j(t)$. If:
\begin{enumerate}
\item $i_j$ are constant, or
\item $i_j$ are immersions and $II_{{\cal H}_1}(v,v)\leq II_{{\cal H}_2}(w,w)$ for all $v\in di_1(T_{x_0}{\cal H}_1), w\in di_2(T_{\overline{\i}(x_0)}{\cal H}_2)$ with $|v|=|w|$,
\end{enumerate}
then the second fundamental forms of the graphs satisfy
$$
II_1(V,V)\geq II_2(W,W)
$$
for all tangent vectors $V,W$ such that $|V|=|W|$ and $\langle V,\gamma'_1(t_0)\rangle=0=\langle W,\gamma'_2(t_0)\rangle$.

In particular, every principal curvature of the graph $\psi_1$ at $\varphi_1(x_0,t_0)$ is greater than or equal to every principal curvature of the graph $\psi_2$ at $\varphi_2(\overline{\i}(x_0),t_0)$.
\end{teo}
\begin{proof}
Let $V\in T_{\gamma_1(t_0)}\m_1$ and $W\in T_{\gamma_2(t_0)}\m_2$ with $$|V|=1=|W|\quad \text{ and }\quad  g_1(V,\gamma_1'(t_0))=0=g_2(W,\gamma_2'(t_0)).$$ As above, $V$ and $W$ are identified as tangent vectors to the graphs at the points $\psi_j(\gamma_j(t_0),f(t_0))$, $j=1,2$, respectively.

There exist unique perpendicular Jacobi fields $J_j:[0,t_0]\fl T\m_j$, $j=1,2$, along the geodesics $\gamma_1$ and $\gamma_2$ respectively, such that $J_1(t_0)=V$, $J_2(t_0)=W$ with the additional property: $J_j(0)=0$ if $i_j$ is constant, or $-dn(J_j(0))+\frac{DJ_j}{dt}(0)$ is proportional to $\gamma_j'(0)$ if $i_j$ is an immersion (see, for instance, \cite{Wa}).

Let $\{e_k(t)\}_{k=1}^m$ be an orthonormal basis of parallel vector fields along $\gamma_1$, which are perpendicular to $\gamma_1'(t)$, with $m=dim(\m_1)-1$, and such that $e_1(t_0)=J_1(t_0)$. In a similar way, let $\{b_k(t)\}_{k=1}^n$ be an orthonormal basis of parallel vector fields along $\gamma_2$, which are perpendicular to $\gamma_2'(t)$, with $n=dim(\m_2)-1$, and such that $b_1(t_0)=J_2(t_0).$

We define the functions $a_k(t)$ as the ones given by the equality
$$
J_1(t)=\sum_{k=1}^m a_k(t) e_k(t).
$$
And consider a new vector field along $\gamma_2(t)$ given as
$$
W(t)=\sum_{k=1}^m a_k(t) b_k(t).
$$

Since $W(t_0)=J_2(t_0)$ the minimizing property of the Jacobi fields (see \cite{Wa}) gives us
$$
{\cal I}_{(t_0,{\cal H}_2)}(J_2,J_2)\leq {\cal I}_{(t_0,{\cal H}_2)}(W,W).
$$
Hence,
$$
II_{2}(J_2(t_0),J_2(t_0))=\frac{f'(t_0)}{\sqrt{1+f'(t_0)^2}}\ {\cal I}_{(t_0,{\cal H}_2)}(J_2,J_2)\leq \frac{f'(t_0)}{\sqrt{1+f'(t_0)^2}}\ {\cal I}_{(t_0,{\cal H}_2)}(W,W).
$$

Since $|J_1(t)|=|W(t)|$, $|\frac{DJ_1}{dt}(t)|=|\frac{DW}{dt}(t)|$ and $II_{{\cal H}_1}(J_1(0),J_1(0))\leq II_{{\cal H}_2}(W(0),W(0))$ when $i_j$ is an immersion, then we obtain from (\ref{e1}), (\ref{e2}), (\ref{e3}) and the previous inequality that
$$
II_{2}(J_2(t_0),J_2(t_0))\leq II_{1}(J_1(t_0),J_1(t_0)),
$$
as we wanted to show.
\end{proof}
A different proof of this result was given in \cite{GL} when $dim(\m_1)=dim(\m_2)=2$ as well as many applications.

Theorem \ref{t1} gives us a criterium for comparing all the principal curvatures of a graph at a point with all the principal curvatures of another graph at the corresponding point. Now, we look for some weaker conditions in order to compare the mean curvature of both graphs.

\begin{teo}\label{t2}
In the previous conditions assume $dim(\m_1)=dim(\m_2)$ and the metric of $\m_1$ can be written as
\begin{equation}\label{model}
g_1=dt^2+G(t)g_0,
\end{equation}
where $g_0$ is the (n-1)-dimensional metric of a space form. If $Ric^1(\gamma_1'(t))\leq Ric^2(\gamma_2'(t))$, where $Ric^j(\gamma_j'(t))$ denotes the Ricci curvature in the direction of the unit vector $\gamma_j'(t)$, and
\begin{enumerate}
\item $i_j$ are constant, or
\item $i_j:{\cal H}_j\fl\m_j$ are immersions and their mean curvatures $H_{{\cal H}_j}$ satisfy $H_{{\cal H}_1}(x_0)\leq H_{{\cal H}_2}(\overline{\i}(x_0))$,
\end{enumerate}
then the mean curvatures $H_j$ of the graphs in $\m_j\times\r$ satisfy
$$
H_1(\gamma_1(t_0))\geq H_2(\gamma_2(t_0)).
$$
\end{teo}

\begin{proof} In order to compare the mean curvatures of the graphs given by (\ref{psi2}), we use the trace of the second fundamental forms at points $\psi_1(x_0, t_0)$ and $\psi_2(\overline{\i}(x_0), t_0)$. For this, from (\ref{e3}) and the fact that the function $f(t)$ is increasing, it is sufficient to compare the corresponding sums of the index forms.\\

Let $\{e_k(t)\}_{k=1}^{n-1}$ be an orthonormal basis of parallel fields along $\gamma_1(t)$, orthogonal to $\gamma'_1(t)$, and let $\{b_k(t)\}_{k=1}^{n-1}$ be an orthonormal basis of parallel fields along $\gamma_2(t)$, orthogonal to $\gamma'_2(t)$ with $t \in [0, t_0]$. Then, there exist unique perpendicular Jacobi fields $J_k^j:[0,t_0]\longrightarrow TM_j$, with $j=1,2$ and $k=1,...,n-1$ along $\gamma_1(t)$ and $\gamma_2(t)$ respectively, such that
\begin{enumerate}
  \item $J_k^1(t_0)=e_k(t_0), \quad J_k^2(t_0)=b_k(t_0), \quad k=1,...,n-1.$
  \item $J_k^j(0)=0 $ if $i_j$ are constant maps, and so $$\sum_{k=1}^{n-1} g_j(J_k^j(0),\frac{DJ_k^j}{dt}(0))=0, \quad j=1,2; \quad \mbox{or}$$
   $-dn(J_k^j(0))+\frac{DJ_k^j}{dt}(0)$ is proportional to $\gamma'_j(0) $ if $i_j$ are immersions (see, for instance, \cite{Wa}), and so $$\sum_{k=1}^{n-1}g_1(J_k^1(0),\frac{DJ_k^1}{dt}(0))=-(n-1)\,H_{\mathcal{H}_1}(x_0),$$ and
   $$\sum_{k=1}^{n-1}g_2(J_k^2(0),\frac{DJ_k^2}{dt}(0))=-(n-1)\,H_{\mathcal{H}_2}(\overline{\i}(x_0)).$$
\end{enumerate}

From (\ref{model}) the previous Jacobi fields $J_k^1$ in $\m_1$ satisfy
$$|J_k^1(t)|=|J_i^1(t)| \quad \mbox{and} \quad |\frac{DJ_k^1}{dt}(t)|=|\frac{DJ_i^1}{dt}(t)| ,\qquad \mbox{with} \quad i, k = 1,...,n-1.$$

Now, we define the functions $a_{ik}$ as the ones given by the equalities
$$J_i^1(t)=\sum_{k=1}^{n-1} a_{ik}(t) \,e_k(t), \quad i=1,...,n-1.$$
Consider the new vector fields $W_i(t)$ along $\gamma_2(t)$ given by
$$W_i(t)=\sum_{k=1}^{n-1} a_{ik}(t) \,b_k(t), \quad i=1,...,n-1.$$
In these conditions, $a_{ik}(t_0)=\delta_{ik}$ and $W_i(t_0)=J_i^2(t_0)$. Moreover, by construction,
$$|J_i^1(t)|=|W_i(t)| \quad \mbox{and} \quad |\frac{DJ_i^1}{dt}(t)|=|\frac{DW_i}{dt}(t)|.$$

Observe now that if $i_j$ are constant maps
$$\sum_{i=1}^{n-1} g(J_i^2(0),\frac{DJ_i^2}{dt}(0))= \sum_{i=1}^{n-1} g(W_i(0),\frac{DW_i}{dt}(0))= \sum_{i=1}^{n-1} g(J_i^1(0),\frac{DJ_i^1}{dt}(0)) = 0.$$
On the other hand, if the maps $i_j$ are immersions, then $$\sum_{k=1}^{n-1}g(J_k^2(0),\frac{DJ_k^2}{dt}(0))=-(n-1)\,H_{\mathcal{H}_2}(\overline{\i}(x_0)) \leq -(n-1)\,H_{\mathcal{H}_1}(x_0) = \sum_{k=1}^{n-1}g(J_k^1(0),\frac{DJ_k^1}{dt}(0)).$$

In addition, the fields ${W_i(t)}$ are also orthogonal on $[0,t_0]$ by construction, and $|W_i(t)|=|W_1(t)|$, $i=1,\ldots,n-1$. Hence
$$-\sum_{i=1}^{n-1}g_2(R(W_i,\gamma_2')\gamma_2',W_i)  = -(n-1)
|W_1|^2\,Ric^2(\gamma_2'(t))\leq$$ $$\leq -(n-1)|J_1|^2\,Ric^1(\gamma_1'(t)) =-\sum_{i=1}^{n-1}g_1(R(J_i^1,\gamma_1')\gamma_1',J_i^1),$$



In these conditions, and by the minimizing property of the Jacobi fields, it is obtained

$$\sum_{i=1}^{n-1} {\cal I}_{(t_0,{\mathcal{H}_2})}(J_i^2(t_0),J_i^2(t_0))=$$
$$=\sum_{i=1}^{n-1} g_2(J_i^2(0),\frac{DJ_i^2}{dt}(0))+\int_0^{t_0}\left(\sum_{i=1}^{n-1}\left|\frac{DJ_i^2}{dt}\right|^2-\sum_{i=1}^{n-1}g_2(R(J_i^2,\gamma_2')\gamma_2',J_i^2)\right)\ dt$$
$$ \leq \sum_{i=1}^{n-1} g_2(W_i(0),\frac{DW_i}{dt}(0))+\int_0^{t_0}\left(\sum_{i=1}^{n-1}\left|\frac{DW_i}{dt}\right|^2-\sum_{i=1}^{n-1}g_2(R(W_i,\gamma_2')\gamma_2',W_i)\right)\ dt$$
$$ \leq \sum_{i=1}^{n-1} g_1(J_i^1(0),\frac{DJ_i^1}{dt}(0))+\int_0^{t_0}\left(\sum_{i=1}^{n-1}\left|\frac{DJ_i^1}{dt}\right|^2-\sum_{i=1}^{n-1}g_1(R(J_i^1,\gamma_1')\gamma_1',J_i^1)\right)\ dt$$
$$=\sum_{i=1}^{n-1} {\cal I}_{(t_0,{\mathcal{H}_1})}(J_i^1(t_0),J_i^1(t_0))$$
as we wanted to show.
\end{proof}

\begin{no}
Observe that a manifold $\m_1$ whose metric is described by (\ref{model}) in geodesic polar coordinates is classically known as a {\it model manifold} (see \cite{GW}).
\end{no}

\section{Existence of barriers in $\m^n \times \r$.}\label{s4}

Our comparison results will allow us to extend some results only known for $\m^n\times\r$  when $\m^n$ is a space form to general ambient spaces $\m^n\times\r$. Thus, in this section we will follow our approach in \cite{GL} for obtaining some existence and non existence results for hypersurfaces in $\m^n\times\r$.

From now on, we denote by $\m^n (c)$ the complete simply connected n-dimensional space form of constant curvature $c$, that is a hyperbolic space if $c<0$, the Euclidean space if $c=0$ or a sphere if $c>0$. Let $s_{c,n}= \frac{n-1}{n} \sqrt{-c}$ be the infimum of the mean curvature of the topological spheres of constant mean curvature in $\m^n(c) \times \r$ when $c<0$. Also, for each $H_0>0$ ($H_0 > s_{c,n}$ if $c<0$) we denote by $r_{c,n}(H_0)$ the radius of the topological sphere of constant mean curvature $H_0$ in $\m^n(c) \times \r$, and for each $K_0>0$ we will denote by $r_{c,n}^*(K_0)$ the radius of the topological sphere of constant Gauss-Kronecker curvature $K_0>0$ in the same ambient space.

Let us start with a topological sphere $S$ of constant mean curvature $H_0$ in $\m^n (c)\times \r$ (see, for instance, \cite{AR, AEG, BE2, HH, PR}). Observe that $S$ is unique up to isometries of the ambient space and only exists for $H_0>s_{c,n}$ if $c<0$. Moreover, $S$ is rotational with respect to a vertical axis and symmetric with respect to a horizontal slice. In particular, $S$ is a bigraph over a geodesic ball of $\m^n (c)$ of radius $r_{c,n}(H_0)>0$.

Thus, let $p\in \m^n (c)$ and $(x,t)$ be geodesic polar coordinates around $p$. Since $S$ is a rotational surface, the lower part of $S$ can be considered as a graph over the geodesic ball centered at $p$ and radius $r_{c,n}(H_0)$, with height function $h(t)$ which only depends on the distance function $t$ to the point $p$. Moreover, $h(t)$ is strictly increasing. Hence, this part of the  hypersurface $S$ of constant mean curvature can be described as
$$ \psi_1(x,t)=(x,t,h(t)) \in \m^n (c) \times \r. $$
Note that, for convenience, we have deleted the parametrization $\varphi$ (given by (\ref{eliminado})) in the previous expression.

Now, given  an n-dimensional Riemannian manifold $\m^n$  and geodesic polar coordinates $(x,t)$ around a point $q\in\m^n$, which are well defined for $0<t\leq r_{c,n}(H_0)$, we can consider the new immersion
$$\psi_2(x,t)=(x,t,h(t))\in\m^n \times \r.$$

Applying the same process for the upper part of $S$, we obtain a sphere $S^{\ast}$ in $\m^n \times \r$ which is a bigraph over the geodesic ball of radius $r_{c,n}(H_0)$ centered at $q$.

We remark that $S^{\ast}$ is symmetric with respect to a horizontal slice as $S$, and any vertical translation of $S^{\ast}$ is congruent to $S^{\ast}$. However, $S^{\ast}$ depends strongly on the point $q\in\m^n$, i. e. if we start with another point $\widetilde{q}\in\m^n$ and obtain a new surface $\widetilde{S^{\ast}}$ following the same process then $S^{\ast}$ and $\widetilde{S^{\ast}}$ are not isometric in general. If the Ricci curvature in the radial directions on $B_r(q) \subset \m$ are greater than or equal to $c$, for all the geodesics $\gamma(t)$ in $\m$ emanating from $q$, using Theorem 2, we have that the mean curvature of $S^{\ast}$ satisfies $H(S^{\ast}) \leq H_0$.
With all of this we obtain

\begin{teo}\label{t3}
Let $B_r$ be a closed geodesic ball of radius $r>0$ in an n-dimensional Riemannian manifold $\m^n$, and $c$ the minimum of the Ricci curvature in the radial directions of unit vectors $\gamma'(t)$ on $B_r$, for all the geodesics $\gamma(t)$ in $\m^n$ emanating from the center of $B_r$. Consider $H_0>0$ such that $r_{c,n}(H_0) = r$. Then, there is no vertical graph over $B_r$ with minimum of its mean curvature satisfying $min(H)\geq H_0$.
\end{teo}
\begin{proof}Assume $\Sigma$ is a graph over $B_r$ with $min(H)\geq H_0$ for a unit normal $N$. Without loss of generality, we  assume that the unit normal $N$ points upwards.

Let $q\in\m^n$ be the center of the geodesic disk $B_r$ and consider the sphere $S^{\ast}$ centered at $q$ previously obtained, which has mean curvature smaller than or equal to $H_0$ for its inner normal.

Move the sphere $S^{\ast}$ up until $\Sigma$ is below $S^{\ast}$, and go down until $S^{\ast}$ intersects $\Sigma$ for the first time. Then, the classical maximum principle for mean curvature asserts that both surfaces must agree locally. In particular, $\Sigma$ and $S^{\ast}$ have constant mean curvature $H_0$ and $\Sigma$ agrees with the lower hemisphere of $S^{\ast}$. However, this is a contradiction because $S^{\ast}$ is not a strict graph over the boundary of $B_r$ since its unit normal is horizontal at those points.
\end{proof}

\begin{teo}\label{t4}
Let $B_r$ be a closed geodesic ball of radius $r>0$ in an n-dimensional Riemannian manifold $\m^n$, $c:= min \{ K_p (\pi) : \partial_t \in \pi, p \in B_r \}$ be the minimum of the radial sectional curvatures on $B_r$, and $K_0>0$ such that $r_{c,n}^*(K_0) = r$. Then, there is no vertical graph over $B_r$ with minimum of its Gauss-Kronecker curvature satisfying $min(K)\geq K_0$ and a point with definite second fundamental form.
\end{teo}
\begin{proof}
The proof follows the same process that in Theorem \ref{t3}, taking now a sphere with constant Gauss-Kronecker curvature in $\m^n(c) \times \r$ (see, for instance, \cite{EGR,ES}), and using Theorem \ref{t1}. The requirement of the graph of having a point with definite second fundamental form is now needed for using the maximum principle. 
\end{proof}

\begin{no}
It should be observed that a similar result to Theorem \ref{t4} is possible for any r-mean curvature $H_r$, with $2\leq r\leq n$, and not only for the Gauss-Kronecker curvature $H_n$.
\end{no}

\begin{teo}\label{t5}
Let $\m^n$ be a complete, simply connected Riemannian manifold with injectivity radius $i>0$ and $c \in \r$ the infimum of its Ricci curvature on $\m^n$. Consider a properly embedded hypersurface $\Sigma$ in $\m^n \times \r$ with mean curvature $H \geq H_0 >0$, ($H_0 > s_{c,n}$ if $c<0$). If $r_{c,n}(H_0)<i$ then the mean convex component of $\Sigma$ cannot contain a closed geodesic ball in $\m^n \times \r$ of radius greater than or equal to the extrinsic semi-diameter of a sphere with constant mean curvature $H_0$ in $\m^n(c) \times \r$.
\end{teo}
\begin{proof} The proof is a consequence of the maximum principle and follows the same process that \cite[Theorem\,2]{GL}, using now Theorem \ref{t2}.
\end{proof}

Also observe that a weaker version of Theorem \ref{t5} is possible for the r-mean curvatures $H_r$, $2\leq r\leq n$, using Theorem \ref{t1}.

Let $\m^n$ be a Hadamard manifold, that is, a complete simply connected Riemannian manifold with non-positive sectional curvature. Since its injectivity radius is $i=\infty$, we obtain as a consequence of the previous result:

\begin{cor} \label{c1}
Let $H_0>0$ and $\m^n$ be a Hadamard manifold with infimum of its Ricci curvature $c>-\infty$. Then, there exists no entire horizontal graph in $\m^n \times \r$ with mean curvature $H \geq H_0 > s_{c,n}$.
\end{cor}

Let us denote by $S_{\frac{n-1}{n}}$ the simply connected rotational entire vertical graph with constant mean curvature $H=\frac{n-1}{n}$ in $\h^n \times \r$. This graph has been described in \cite{BE2}. Again, as a consequence of our comparison results, if we consider the corresponding entire vertical graph $S_{\frac{n-1}{n}}^{\ast}$ in $\m^n\times\r$, we have:

\begin{cor} \label{c2}
Let $\m^n$ be an n-dimensional Hadamard manifold with Ricci curvature smaller than or equal to $-1$. Assume $\Sigma$ is an immersed hypersurface in $\m^n \times \r$ with mean curvature $H \leq \frac{n-1}{n}$ and cylindrically bounded vertical ends. Then $\Sigma$ must have more than one end.
\end{cor}

We obtain now a generalization to $\m^n\times\r$ of a theorem proven in \cite{BE2} for the product space $\h^n\times\r$.

\begin{teo}\label{t7}
Let $\m^n$ be an n-dimensional Hadamard manifold with sectional curvature pinched between  $-c^2$ and $-1$, for a constant $c \geq 1$. Let $\Omega$ be a bounded domain in $\m^n \times \{0\}$, with boundary given by a compact embedded hypersurface  $\Gamma$. Assume all the principal curvatures of $\Gamma$ are greater than $c$, then for any $H_0 \in [0, \frac{n-1}{n}]$ there exists a graph $h$ over $\Omega$ with constant mean curvature $H_0$ and zero boundary data.

Moreover, if $\Sigma$ is a compact hypersurface immersed in $\m^n \times \r$ with boundary $\Gamma$ and constant mean curvature $H_0$ then, up to a symmetry with respect to $\m^n \times \{0\}$, $\Sigma$ agrees with the previous graph.
\end{teo}
\begin{proof}
Observe that $\Omega$ must be a convex bounded domain in $\m^n$ and homeomorphic to a ball, and $\Gamma$ must be homeomorphic to a sphere (see \cite{Alex}). 

Let $m_0$ be the minimum of the principal curvatures of $\Gamma$. Since $m_0>c$ we can take a radius $R_0$ big enough such that for every $R>R_0$ the geodesic spheres of radius $R$ in the hyperbolic space of seccional curvature $-c^2$ have principal curvatures smaller than $m_0$. As the sectional curvature of $\m^n$ is bigger than or equal to $-c^2$, the geodesic spheres in $\m^n$ of radius $R\geq R_0$ have principal curvatures smaller than $m_0$.

Let $p\in\Gamma$ and $\gamma_p(t)$ be the geodesic in $\m^n$ starting at $p$ with initial speed given by the unit normal to $\Gamma$ pointing to $\Omega$. It is clear that the geodesic sphere $S_p(R)\subseteq\m^n$ centered at $\gamma_p(R)$ and radius $R$ is tangent to $\Gamma$ at $p$. Moreover, if $R\geq R_0$ the open geodesic ball bounded by $S_p(R)$ contains a punctured neighborhood of $p\in\Gamma$ because the principal curvatures of $S_p(R)$ are bigger than the principal curvatures of $\Gamma$ at $p$ for the same interior unit normal. 

Let $S_0\subseteq\m^n$ be a geodesic sphere such that $\Gamma$ is contained in the geodesic ball bounded by $S_0$, and the distance from $S_0$ to $\Gamma$ is greater than or equal to  $R_0$. Then, consider the map $G:\Gamma\rightarrow S_0$ defined in the following way: given $p\in\Gamma$ the point $G(p)$ is given by the intersection of the geodesic $\gamma_p(t)$
for $t\geq 0$ with $S_0$. It is well known that the previous intersection is given by a unique point due to the convexity of the geodesic spheres (see, for instance, \cite{Alex}).

Thus, if we denote by $S_p$ the geodesic sphere centered at $G(p)$ passing across $p\in\Gamma$, then we have shown that $S_p$ is tangent to $\Gamma$ at $p$ and a punctured neighborhood of $p$ in $\Gamma$ is contained in the open geodesic ball bounded by $S_p$. In fact, we assert

{\it Claim:} For every $p\in\Gamma$ the closed geodesic ball bounded by $S_p$ contains to $\Gamma$.

Observe that if $G:\Gamma\rightarrow S_0$ is injective then the Claim would be proven. Indeed, if there existed $p_1\in \Gamma$ such that $\Gamma\not\subseteq S_{p_1}$ then  there would be a point $p_2\neq p_1$ such that $d(p_2,G(p_1))\geq d(p,G(p_1))$ for all $p\in\Gamma$. Thus, the geodesic sphere centered at $G(p_1)$ passing across $p_2$ is tangent to $\Gamma$ and contains $\Gamma$ in its interior, and so $G(p_2)=G(p_1)$.

Hence, assume there exist two points $p_1,p_2\in\Gamma$ such that $G(p_1)=G(p_2)$. In such a case, we have shown that $p_1$ and $p_2$ are two strict local maxima for the distance function $\varrho(p)$ from $p\in\Gamma$ to the fixed point $G(p_1)=G(p_2)$. Now, we distinguish two cases depending on the dimension of $\Gamma$:
\begin{enumerate}
\item If dim($\Gamma$)$\geq2$ then we can use the mountain pass lemma for the function $\varrho$ and there must exist a third point $p_3$ which is a saddle point for $\varrho$. Thus, the geodesic sphere $\widetilde{S}_{p_3}$ centered at $G(p_1)$ passing across $p_3$ is tangent to $\Gamma$. Therefore, depending on the orientation of the unit normal to $\Gamma$, we have that $\widetilde{S}_{p_3}=S_{p_3}$ or $\widetilde{S}_{p_3}\cap S_{p_3}=\{p_3\}$. But, this contradicts that $p_3$ is a saddle point, because a punctured neighborhood of $p_3$ is contained in the interior of the geodesic ball bounded by $S_{p_3}$.
\item If dim($\Gamma$)$=1$ then we can consider the two closed arcs $\Gamma_1$ and $\Gamma_2$ of $\Gamma$ joining $p_1$ and $p_2$. Since $p_1$ and $p_2$ are strict local maxima for the function $\varrho$, there must exist $p_3\in\Gamma_1$ and $p_4\in\Gamma_2$ different from $p_1$ and $p_2$ which are local minima for $\varrho$. Assume $\varrho(p_3)\leq \varrho(p_4)$, then from the convexity of $\overline{\Omega}$ the geodesic arc $\Lambda$ joining $p_3$ and $p_4$ is contained in $\overline{\Omega}$. But, $\Lambda\backslash\{p_3,p_4\}$ is contained in the open geodesic ball centered at $G(p_1)$ and radius $\varrho(p_4)$ from the convexity of the geodesic ball. This contradicts that $p_4$ is a minimum for $\varrho$.
\end{enumerate}

Once the Claim is proven, consider  a compact hypersurface $\Sigma$ immersed in $\m^n \times \r$ with boundary $\Gamma$ and constant mean curvature $H_0$. 

Let $S$ be the rotational entire graph with constant mean curvature $H_0$ in $\h^n\times\r$ for its unit normal pointing upwards. Consider a point $p\in\Gamma\subseteq\m^n$ and the associated entire graph $S^{\ast}\subseteq\m^n\times\r$ when we use geodesic polar coordinates at $G(p)\in\m^n$. Up to a vertical translation we can assume that $S^{\ast}\cap\m^n\times\{0\}$ is the geodesic sphere centered at $G(p)$ and containing to $p$ in $\m^n\times\{0\}$. Thus, from the previous Claim, $\Gamma\times\{0\}$ is contained in the closed mean convex component of $S^{\ast}$, and $(p,0)\in S^{\ast}$.

Let us also denote by $\overline{S^{\ast}}$ the reflection of $S^{\ast}$ with respect to the horizontal slice $\m^n\times\{0\}$. The entire graphs $S^{\ast}$ and $\overline{S^{\ast}}$
are congruent, and from Theorem \ref{t1} we obtain that they have mean curvature $H\geq H_0$ for its unit normal pointing to the mean convex component.

As $\Sigma$  is compact we can move vertically $S^{\ast}$ in such a way that $\Sigma$ is completely contained in the mean convex component of $S^{\ast}$. Now, from the maximum principle, if we move back $S^{\ast}$ then the surfaces $\Sigma$ and $S^{\ast}$ do not intersect until $S^{\ast}$ is in its initial position. The same is true for $\overline{S^{\ast}}$.

Therefore, for every $p\in\Gamma$ the hypersurface $\Sigma$ is contained in the compact domain determined by the intersection of the mean convex components of $S^{\ast}$ and $\overline{S^{\ast}}$. In particular, the interior of $\Sigma$ is contained in the solid cylinder $\Omega\times\r$, and if $\Sigma$ was given by the graph of a function $h$ then its height is bounded a priori and so is its gradient at the boundary $\Gamma$.

Now we can prove that there exists a graph $h$ over $\Omega$ with constant mean curvature $H_0 \in [0, \frac{n-1}{n}]$ and zero boundary data. That is, we want to solve the following Dirichlet problem
$$
\left\{\begin{array}{ll}
{\displaystyle div\left(\frac{\nabla h}{\sqrt{1+|\nabla h|^2}}\right)=n\,H_0,}& \qquad\text{in }\Omega\\[2mm]
h=0&\qquad\text{on }\Gamma
\end{array}\right.
$$
where the divergence and gradient $\nabla h$ are taken with respect to the metric
on $\m^n$ (see \cite{Sp}).

We have proven the existence of height estimates and gradient estimates at the boundary. Hence, from \cite{Sp}, we also have global gradient estimates, and the existence of $h$ follows from the classical elliptic theory (see \cite{GT} and \cite{Sp}). 

Finally, we want to show that if $\Sigma$ is a compact hypersurface in $\m^n\times\r$ with constant mean curvature $H_0$ and boundary $\Gamma$, then it is the graph of the previous function $h$ or $-h$.

First, let us observe that $\Sigma$  is a vertical graph. In fact, we have shown that $\Sigma$ is contained in the cylinder $\overline{\Omega}\times\r$ and $\Sigma$ has no interior point in $\Gamma\times\r$. Thus, we can use the maximum principle with respect to horizontal slices from the highest point of $\Sigma$ to the lowest point of $\Sigma$, which proves that $\Sigma$ is a graph. 

Moreover, let $\Sigma_0$ be the graph of $h$ or $-h$ which points in the same direction (upwards or downwards) as $\Sigma$. 
Moving $\Sigma$ vertically upwards until $\Sigma$ and $\Sigma_0$ are disjoint, and coming down again we observe that, from the maximum principle, $\Sigma$ cannot touch $\Sigma_0$ till the boundaries
agree. Hence, $\Sigma$ is above $\Sigma_0$. Repeating the same process, but moving now $\Sigma$ vertically downwards, one has $\Sigma$ is below $\Sigma_0$. Therefore, $\Sigma$ and $\Sigma_0$ agree, as we wanted to show.
\end{proof}
\begin{no}
It is an interesting open question if, under the previous conditions on $\m^n$, there exists an entire vertical graph over $\m^n$ for every constant mean curvature $H_0\in(0,(n-1)/n]$.
\end{no}
\begin{teo}\label{t8}
Let $B_r$ be a closed geodesic ball of radius $r>0$ in $\m^n$, $c:= max \{ K_p (\pi) : \partial_t \in \pi, p \in B_r \}$ be the maximum of the radial sectional curvatures on $B_r$, and $K_0>0$ such that $r_{c,n}^*(K_0) = r$. Then, there  exists a strictly convex graph $h_K$ over $B_r$ of constant Gauss-Kronecker curvature $K>0$ in $\m^n \times \r$ y $h_K |_{\partial B_r}=0$, for any $K<K_0$.
\end{teo}
\begin{proof} Let us consider a sphere $S$ in $\m^n(c) \times \r$ with positive constant Gauss-Kronecker curvature $K<K_0$. From Theorem \ref{t1}, the corresponding sphere $S^{\ast}$ in $\m^n\times\r$, using polar coordinates at the center of $B_r$, has Gauss-Kronecker curvature greater than or equal to $K$, and so it is a subsolution for the existence of the graph we are looking for. Thus, from \cite{Gu} (see also \cite{Sp}) there exists a strictly convex graph of constant Gauss-Kronecker curvature $K$ and zero boundary data.
\end{proof}

\end{document}